\documentclass[11pt, reqno]{amsart}
\usepackage{euscript,amscd,amsgen,amsfonts,amssymb,latexsym,graphicx,psfrag,pgfplots, multicol}

\newcommand{\eqdef}{\stackrel{\scriptscriptstyle\rm def}{=}}

\newtheorem{theorem}{Theorem}

\newtheorem{proposition}{Proposition}
\newtheorem{corollary}{Corollary}

\newtheorem{lemma}{Lemma}

\newtheorem{example}{Example}

\newtheorem*{thma}{Theorem}

\newcommand{\beha}{\begin{enumerate}}
\newcommand{\behe}{\end{enumerate}}
\renewcommand{\epsilon}{\varepsilon}

\newcommand{\Or}{\mathcal{O}}
\newcommand{\C}{\mathcal{C}}

\newcommand{\cM}{\EuScript{M}}

\newcommand{\cH}{\EuScript{H}}
\newcommand{\cV}{\EuScript{V}}
\newcommand{\bR}{{\mathbb R}}

\newcommand{\bN}{{\mathbb N}}

\newcommand{\cA}{{\mathcal A}}

\newcommand{\cC}{{\mathcal C}}

\newcommand{\cG}{{\mathcal G}}
\newcommand{\cR}{{\mathcal R}}

\def\1{1\!\!1}
 \def\tY{\widetilde{Y}}
\def\and{\text{ and }}

        \def\conv{\text{{\rm conv}}}

\def\tC{\tilde{\cC}}

\def\tU{\widetilde{U}}

                        \def\^{\tilde}

\def\Per{{\rm Per}}

\def\Per{{\rm Per}}

\def\1{1\!\!1}

\def\rv{{\rm rv}}

\DeclareMathSymbol{\varnothing}{\mathord}{AMSb}{"3F}
\renewcommand{\emptyset}{\varnothing}
\title{A shift map with a discontinuous   entropy function}

\author{Christian Wolf}\address{Department of Mathematics, The City College of New York, New York, NY, 10031, USA}\email{cwolf@ccny.cuny.edu}

\begin{document}

\begin{abstract}
Let $f:X\to X$ be a continuous map on a compact metric space with finite topological entropy.
Further, we assume  that the entropy map $\mu\mapsto h_\mu(f)$  is upper semi-continuous. 
It is well-known that this  implies the continuity of the localized entropy function of a given continuous  potential $\phi:X\to \bR$. In this note we show that this result does not carry over to the case of higher-dimensional potentials $\Phi:X\to \bR^m$. Namely, we construct   for a shift map $f$ a $2$-dimensional Lipschitz continuous potential $\Phi$   with a discontinuous localized entropy function.
\end{abstract}
\keywords{Rotation set, localized entropy,  entropy spectrum, multifractal analysis, discontinuity}
\subjclass[2010]{37A35, 37B10,37C40}
\maketitle

\section{Introduction}
Let $f:X\to X$ be a continuous map on a compact metric space with finite topological entropy, and let $\cM$ denote the set of all $f$-invariant Borel probability measures on $X$ endowed with the weak$^\ast$ topology.  This makes $\cM$ a compact convex metrizable topological space.
For a continuous $m$-dimensional potential $\Phi=(\phi_1,\cdots,\phi_m):X\to \bR^m$ we define 
\begin{equation}
\cR(\Phi)=\{\rv(\mu):\mu\in \cM\},
\end{equation}
 where $\rv(\mu)=(\int \phi_1\, d\mu,\cdots, \int \phi_m\, d\mu)$. It follows that $\cR(\Phi)$ is a compact and convex subset of $\bR^m$. The set $\cR(\Phi)$  is frequently referred to as the
rotation set of $\Phi$ (see e.g. \cite{Bl,GL,Je,KW,KW4,Z}), while in the context of multifractal analysis it is
often referred to as the spectrum of (Birkhoff) ergodic averages (see e.g. \cite{BPS, BSS,C}). The localized entropy function of $\Phi$ on $\cR(\Phi)$ is defined 
by 
\begin{equation}\label{defcH}
\cH(w)=\cH_\Phi(w)=\sup \{h_\mu(f): \rv(\mu)=w\},
\end{equation}
where $h_\mu(f)$ denotes the measure-theoretic entropy of $\mu$.  We note that for various systems and potentials the localized entropy function  coincides with the entropy of certain multifractal level  sets (e.g. \cite{BSS, C}).
Recall that the measure-theoretic entropy is an affine function on $\cM$. This shows that $w\mapsto\cH(w)$ is   concave which implies its continuity on the  interior of $\cR(\Phi)$, see e.g. \cite{R}. If $\cR(\Phi)$ has empty interior  we still obtain the continuity of $\cH$  on the relative interior of $\cR(\Phi)$, i.e., the interior of $\cR(\Phi)$ considered as a subset of the affine hull of $\cR(\Phi)$. Another frequently considered condition is the upper semi-continuity of the entropy map 
$\mu\mapsto h_\mu(f),$
which holds for example when  $f$ is expansive \cite{Wal:81}, when $f$ is a $C^\infty$-map on a compact smooth Riemannian manifold \cite{N} or when $f$ satisfies entropy-expansiveness (as for example certain partial hyperbolic systems \cite{KK}).
 The upper semi-continuity of the entropy map immediately implies that the supremum in \eqref{defcH} is actually a maximum and  more importantly that $w\mapsto \cH(w)$ is upper semi-continuous. 
 One might suspect that the latter actually even guarantees the continuity of the localized entropy function for all dimensions $m$. Indeed, it was stated by Jenkinson \cite[p. 3723]{Je} that the upper-semi continuity of the entropy map implies  the continuity of the localized entropy. This claim was restated by Kucherenko and Wolf in \cite{KW,KW2,KW4}.\footnote{We note that the theorems in \cite{Je,KW,KW2,KW4} do not rely on the continuity of the localized entropy function. The only exception is Theorem A in \cite{KW4} whose proof uses the continuity of $\cH$ restricted to a line segment, i.e., $m=1$. As noted above, for $m=1$ the localized entropy is always continuous.}
However, it turns out that the argument in \cite{Je} is incomplete. While the continuity of every upper semi-continuous concave function with domain in $\bR$  is  immediate, the situation in higher dimensions   is more delicate. Indeed, a striking result by Dale, Klee and Rockafellar \cite{GKR} shows  that for a  compact convex set $D\subset \bR^m$ the property that \emph{every} concave upper semi-continuous function on $D$ is continuous is equivalent to $D$ being a polyhedron.\footnote{We note that the results in \cite{GKR} are formulated in terms of lower semi-continuous convex functions.} 
We point out that $\cR(\Phi)$ being a polyhedron actually occurs in relevant
situations, e.g. for subshifts of finite type (SFT) and locally constant potentials in \cite{Je,Z}, and  for certain non-locally constant potentials in \cite{Je,KW}. On the other hand, the results in \cite{GKR} do not imply that $w\mapsto \cH(w)$ can be discontinuous. After all $\cH$ is a rather special upper semi-continuous concave function.
In this note we show that the continuity of the localized entropy function can even fail  in the case of  shift maps and Lipschitz continuous potentials. More precisely, we have the following result (see Example 1 and Theorem 1 in the text).

\begin{thma}
Let $f:X\to X$ be a shift map on a one-sided full shift with $3$ symbols. Then there exists a Lipschitz continuous potential $\Phi:X\to \bR^2$ with the following properties:
\begin{enumerate}
\item[(i)] The set $\cR(\Phi)$ has non-empty interior and countably many extreme points of which all but one are isolated;
 \item[(ii)] The localized entropy function $w\mapsto \cH(w)$ is discontinuous at the non-isolated extreme point.
 \end{enumerate}
\end{thma}

Further, one can  show that the localized entropy function in the theorem is analytic  on the interior of $\cR(\Phi)$. This follows from a more general analyticity result for so-called STP-maps (including SFT's, uniformly hyperbolic systems and expansive homeomorphisms with specification) and for H\"older continuous potentials, see \cite{BSS, GK,KW}. 
We  note that the reason for formulating our theorem for one-sided shift maps on a shift space with  $3$ symbols   is for the  ease  of presentation. Our techniques can be applied to obtain similar discontinuity results for more general SFT's in the one-sided and two-sided case.

We end the introduction with the discussion of  a simple example of a upper semi-continuous concave function that fails to be continuous. Let $\cR=\{(x_1,x_2)\in  \bR^2:  x_1^2\leq x_2 \leq 1\}$, and  define $g:\cR\to \bR$ by $g(x)=-\frac{x_1^2}{x_2}+1$ for $x_2>0$ and $g(0,0)=1$. It is straight-forward to verify that $g$ is  concave and $g(\cR)=[0,1]$. Further, $g$ is continuous everywhere except at $(0,0)$ where $g$ is only upper semi-continuous. The limit of $g(x_1,x_2)$ is $0$ as $(x_1,x_2)$ approaches $(0,0)$ along the parabola $x_2=x_1^2$. However, the limit is $1$ when $(x_1,x_2)$ approaches $(0,0)$ along any line segment in $\cR$. Moreover, $g$ attains in each neighborhood of $(0,0)$ all values in $[0,1]$. Indeed, if $S_{x}$ denotes a line segment joining  a point $x$ on the parabola $x_2=x_1^2$ and $(0,0)$ then $g(S_x)=[0,1]$. We note that while the function $g$ is not lower semi-continuous, it does attain its infimum. We point out that there do exist bounded  functions that are concave and upper semi-continuous but  do not attain their infima, see \cite[Lemma 1]{GKR}.

This paper is organized as follows. In Section 2 we recall some basic notation from symbolic dynamics and then construct  in Section 3 an example of a discontinuous localized entropy function. The main ingredients of the proof are presented in Proposition \ref{lem2} and Theorem \ref{thmfin}.

\section{Shift maps}\label{sec:shiftmaps}
We collect some basic notation and facts for shift maps. Let $d\in \bN$, and let $\cA=\{0,\dots,d-1\}$ be a finite alphabet with $d$ symbols. The (one-sided) {\em shift space} $X=X_d$ on the alphabet $\cA$ is the set of
all sequences $\xi=(\xi_k)_{k=1}^\infty$ where $\xi_k\in \cA$ for all $k\in \bN$.  We endow $X$ with the {\em Tychonov product }topology
which makes $X$ a compact metrizable space. For example, given $0<\theta<1$, the metric given by
\begin{equation}\label{defmet}
d(\xi,\eta)=d_\theta(\xi,\eta)\eqdef\theta^{\min\{k\in \bN:\  \xi_k\not=\eta_k\}}\qquad\text{and}\qquad d(\xi,\xi)=0
\end{equation}
induces the Tychonov product topology on $X$.
The {\em shift map} $f:X\to X$, defined by $f(\xi)_k=\xi_{k+1}$, is a continuous $d$-to-$1$ map on $X$. Let $\cM$ be the set of all invariant Borel probability measures endowed with the weak$^\ast$ topology, and let $\cM_E\subset \cM$ denote the subset of ergodic measures. Recall that $\cM$ is a compact convex metrizable topological space. Given $\mu\in \cM$ we denote by $h_\mu(f)$ the measure-theoretic entropy of $\mu$, see  \cite{Wal:81} for the definition and details.
Clearly $f$ is expansive and consequently the entropy map $\mu\mapsto h_\mu(f)$ is upper semi-continuous.
We say $t=t_1t_2 \cdots t_k\in \cA^k$ is a block of length $k$  and write $|t|=k$. Further, $\epsilon$  denotes the empty  block.
 Moreover, we say $s=s_1s_2\cdots s_l$ is a subblock of $t$ if there exists $1\leq i\leq k$ with $i+l-1\leq k$ such that $s_1=t_i,s_2=t_{i+1},\cdots, s_l=t_{i+l-1}$. 
 Given $\xi\in X$, we write $\pi_k(\xi)=\xi_1\cdots \xi_k\in \cA^k$.
 For $\xi_i\in \cA$ and $k\in \bN$ we write $\xi_i^k=\xi_i\cdots \xi_i\in \cA^k$ and define the concatenation of blocks $t$ and $s$ by $ts=t_1\cdots t_ks_1\cdots s_l    $. Moreover, we denote by $t^k$ the $k$-times concatenation of the bock $t$.
We denote the  cylinder of length $k$ generated by $t$ by $\cC_k(t)=\{\xi\in X: \xi_1=t_1,\dots, \xi_k=t_k\}$.  Given $\xi\in X$ and $k\in \bN$, we call $\cC_k(\xi)=\cC(\pi_k(\xi))$ the  cylinder of length $k$ generated by $\xi$. 
Further, we call  $\Or(t)=t_1\cdots t_kt_1\cdots t_kt_1\cdots t_k\cdots \in X$
the  periodic point with period $k$ generated by $t$. 
We denote by  $\Per_n(f)$ the set of periodic points of $f$ with prime period $n$  and by $\Per(f)$ and the set of periodic points of $f$. Let $x\in \Per_n(f)$.  We call $\tau_x=x_1\cdots x_{n}$ the generating segment of  $x$, that is $x=\Or(\tau_x)$.  
For $x\in \Per_n(f)$,  the unique invariant measure supported on the orbit of $x$ is given by 
\begin{equation}\label{perm}
\mu_x=\frac{1}{n}(\delta_x+\dots +\delta_{f^{n-1}(x)}), 
\end{equation} where
 $\delta_y$ denotes the  Dirac measure on $y$. We also call $\mu_x$ the periodic point measure of $x$. Obviously, $\mu_x=\mu_{f^l(x)}$ for all $l\in\bN$. We write $\cM_{\rm Per}=\{\mu_x: x\in \Per(f)\}$ and observe that $\cM_{\rm Per}\subset \cM_E$.

\section{Construction of the example.} \label{sec:7}
In this section we give an example of a shift map and a $2$-dimensional Lipschitz continuous potential   that exhibits a discontinuous localized entropy function.
For convenience we consider here a one-sided shift map on a shift space with 3 symbols. We note that our construction can be modified to obtain discontinuous localized entropy functions for  other shift maps  with positive entropy.  We begin by constructing a certain compact convex subset of $\bR^2$ that will become $\cR(\Phi)$ in our example.

Fix $a,b>0$ and fix $\lambda\in\bN$ with $\lambda\geq 3$. Fix $\theta\in (0,1)$. We consider a continuous function $h:[0,a]\to\bR$ which is  strictly increasing and strictly concave. Further assume $h(0)=0$ and $h(a)=b$. Let $(x_k)_{k\in\bN}$ be a strictly decreasing sequence with $x_k\in (0,a)$ for all $k\geq 1$ such that $v_k\eqdef(x_k,h(x_k))$ satisfies
\begin{equation}\label{eq1}
 ||v_k||<C\theta^k
 \end{equation}
for all $k\in \bN$ and some $C>0$. The existence of such a sequence $(x_k)$ follows from the continuity of $h$ at $0$.  We define $u_k=(x_k,0)$ for all $k\in \bN$. Since $||u_k||\leq ||v_k||$, equation \eqref{eq1} also holds for $u_k$.
  Let $w_\infty=(0,0)$ and $w_0=(a,0)$. Further, for $k\geq 1$ we define 
  \begin{equation}\label{eq2}
  w_k=\frac{1}{k+\lambda}\left(\lambda w_0+\sum\limits_{j=1}^{k}v_k\right).
  \end{equation}
Define
$\cV=\{w_k: k\geq 0\}\cup     \{w_\infty\}$.
Further, let $\cR=\conv(\cV)$ denote the convex hull of $\cV$. For $k\geq 1$ let $m_k$ denote the slope of the line segment joining $w_k$ and $w_{k-1}$.
 Since $(x_k)_k$ is strictly decreasing it follows that the $x$-coordinates of $(w_k)_k$ are strictly decreasing. Thus, $m_k\in \bR$ for all $k\geq 1$. We refer to Figure 1 for an illustration.
\begin{proposition}\label{prop000}The set $\cR$ has the following properties:
\begin{enumerate}
\item[(i)]
 $\lim_{k\to\infty} w_k=w_\infty$ and $\cR$ is compact;
\item[(ii)]
The sequence $(m_k)_{k\geq 2}$ is strictly decreasing;
\item[(iii)]
The boundary of $\cR$ is an infinite polygon with extreme point set $\cV$. 
\end{enumerate}
\end{proposition}
\begin{proof}
{\rm (i)} That $\lim_{k\to\infty} w_j=w_\infty$ follows from \eqref{eq1} and \eqref{eq2}.  Hence, $w_\infty$ is the only accumulation point of $\cV$. We conclude that $\cV$ is compact which implies the compactness of its convex hull $\cR$.\\
{\rm (ii)} By \eqref{eq2}, 
\begin{equation}
w_{k+1}=\frac{k+\lambda}{k+1+\lambda} w_k + \frac{1}{k+1+\lambda}v_{k+1}.
\end{equation}
It now follows from an elementary induction argument   that the points $w_k$ lie strictly below the graph of $h$. Therefore, the statement that $m_k$ is strictly decreasing follows from $h$ being strictly increasing.\\
{\rm (iii)}
First notice that $w_0$ and $w_\infty$ are extreme points of $\cR$. This holds since $\cR$ has empty intersection with $\{(x,y): x<0\}$, $\{(x,y): x>a\}$ and $\{(x,y):y<0\}$. Finally, for $k\geq1$  that $w_k$ is an extreme point of $\cR$   follows from statement (ii).
\end{proof}

\begin{figure}[htb]
\begin{center}
\begin{tikzpicture}[scale=.5]
\filldraw[fill=lightgray,fill opacity=.25] (11.6,0) -- (10.,2.7);
\filldraw[fill=lightgray,fill opacity=.25] (10.0,2.7) -- (7.3,4.0);
\filldraw[fill=lightgray,fill opacity=.25] (7.3,4.0) -- (4.3,3.6);
\filldraw[fill=lightgray,fill opacity=.25] (4.3,3.6) -- (2.5,2.8);
\filldraw[fill=lightgray,fill opacity=.25] (2.5,2.8) -- (1.5,2.1);
\filldraw[fill=lightgray,fill opacity=.25] (1.5,2.1) -- (0.9,1.6);
\filldraw[fill=lightgray,fill opacity=.25] (0.9,1.6) -- (0.5,1.2);
\filldraw[fill=lightgray,fill opacity=.25] (0.5,1.2) -- (0.2,0.7);
\filldraw[fill=lightgray,fill opacity=.25] (0.2,0.7) -- (0.1,0.6);
\filldraw[fill=lightgray,fill opacity=.25] (0.1,0.6) -- (0,0);

\draw[->] (0,0) -- (13.3,0) node[below right]{$x$};
\draw[->] (0,0) -- (0,10.5) node[above right]{$y$};
\draw[domain=0:3.445751311064591,smooth,variable=\x] plot ({\x*\x},{2.9*\x}) node[right]{$h$};
\draw \foreach \x in {11.6,7.3,3.9,1.2,.5,.2,.05}
{({\x},.15) -- ({\x},-.15)};

\node[above] at (10.2,2.6) {$w_1$};
\node[above] at (7.4,3.9) {$w_2$};
\node[above] at (4.4,3.5) {$w_3$};
\node[above] at (2.6,2.85) {$w_4$};

\node[below] at (11.6,-.15) {$a$};
\node[below] at (7.3,-.15) {$x_1$};
\node[below] at (3.9,-.15) {$x_2$};
\node[below] at (1.4,-.15) {$x_3$};
\node[below] at (.5,-.15) {$x_4$};
\draw \foreach \x in {11.6,7.3,3.9,1.2,.5,.2,.1,.05}
{
({\x-.15},{2.9*sqrt(\x)})--(({\x+.15},{2.9*sqrt(\x)})
({\x},{2.9*sqrt(\x)-.15})--(({\x},{2.9*sqrt(\x)+.15})
};
\node at (5.7,1.4) {{\Large $\operatorname{\cR=\cR(\Phi)}$}};
\end{tikzpicture}
\end{center}
\caption{\label{DefPhi}The set $\cR=\cR(\Phi)$ in Example 1.  }
\end{figure}

 \begin{example}\label{ex1}
Let $f:X\to X$ be the one-sided full shift with alphabet $\{0,1,2\}$ endowed with the $\theta$-metric where $\theta$ is as in \eqref{eq1}. We construct a potential  $\Phi$ as follows:   
First, we define several subsets of $X$.  Let $S=\{0,1\}$. We define
\begin{align*}
X(l)&=\{\xi \in X: \xi_1,\dots,\xi_{l-1}\in S, \xi_l=2\},\\
X_0(\lambda)&=\bigcup_{l=1}^\lambda X(l),\\
X(\infty)&=\{\xi \in X: \xi_l\in S \ \mbox{for all} \ l\in \bN\}=S^\bN.
 \end{align*}
Note that $X(1)=\cC_1(2)=\{\xi \in X: \xi_1=2\}$.
We define a potential $\Phi:X\rightarrow\mathbb{R}^2$ by
\begin{equation}\label{defphipot}
\Phi(\xi)=\begin{cases}
w_0& {\rm if}\,\,   \xi\in X_0(\lambda)\\
                      u_{l-\lambda}   & {\rm if}\,\,  \xi\in X(l)\setminus \cC_{l-1}(1^{l-1})\,  , \, l>\lambda\\
                       v_{l-\lambda} & {\rm if} \, \, \xi\in X(l)\cap \cC_{l-1}(1^{l-1}),\, l>\lambda\\
                       w_\infty & {\rm if}\,\,  \xi\in X(\infty)\, 
            \end{cases}
\end{equation}
\end{example}
Throughout the remainder of this paper we study the potential $\Phi$ defined in the Example 1.

\begin{proposition}\label{prop1}
The potential $\Phi$ defined in \eqref{defphipot}
 is Lipschitz continuous and $\cR(\Phi)=\cR$. 
\end{proposition}
\begin{proof}
Let $\xi, \eta\in X$ with $\Phi(\xi)\not= \Phi(\eta)$. First we assume $\xi_k\not= \eta_k$ for some $k\leq \lambda$. Let $C_1=\sup\{||u-v||: u,v\in \Phi(X)\}$. Then 
\begin{equation}\label{eq3}
||\Phi(\xi)-\Phi(\eta)||\leq C_1\leq C_1 \theta^{-\lambda}d(\xi,\eta).
\end{equation}
Next we consider the case $l=\min\{j: \xi_j\not=\eta_j\}>\lambda$. It follows from the definition of $\Phi$ that neither $\Phi(\xi)$ nor $\Phi(\eta)$ belong to $\{u_j,v_j:j=1,\cdots,l-\lambda-1\}\cup\{w_0\}$ since
otherwise $\Phi(\xi)= \Phi(\eta)$.  Therefore, it is sufficient to consider the case  $\Phi(\xi),\Phi(\eta)\in\{u_j,v_j: j\geq l-\lambda\}\cup \{w_\infty\}$. Applying  \eqref{eq1} yields
\begin{equation}\label{eq4}
||\Phi(\xi)-\Phi(\eta)||\leq 2C \theta^{l-\lambda}= 2C\theta^{-\lambda}d(\xi,\eta).
\end{equation}
By combining  \eqref{eq3} and \eqref{eq4} we conclude that $\Phi$ is Lipschitz continuous with Lipschitz constant $\max\{C_1 \theta^{-\lambda}, 2C\theta^{-\lambda}\}$.

Next we prove $\cR(\Phi)=\cR$. Recall that $\cR(\Phi)$ is convex. Therefore, in order to prove  $\cR\subset \cR(\Phi)$ it suffices to show that each extreme point of $\cR$ (i.e. each point in $\cV$) coincides with  the rotation vector of some invariant measure. 
For $k\in\bN$ let  $\xi^k=\Or(1^{k+\lambda-1}2)$, that is $\xi^k$ is the periodic point whose generating segment $\tau_k\eqdef \tau_{\xi^k}$ is given by $k+\lambda-1$ 1's followed by a $2$. Hence $\xi^k\in\Per_{k+\lambda}(f)$. It follows from equations \eqref{perm}, \eqref{eq2} and the definition of $\Phi$ (see \eqref{defphipot}) that $\rv(\mu_{\xi^k})=w_k$. Further, we clearly have
$\rv(\mu_{\Or(02)})=w_0$ and $\rv(\mu_{\Or(0)})=w_\infty$. Hence $\cV\subset \{\rv(\mu_x): x\in \Per(f)\}$ which implies $\cR\subset \cR(\Phi)$.

Finally, we prove $\cR(\Phi)\subset \cR$. It is well-known that the periodic point measures $\cM_{\rm Per}$ are weak$^\ast$ dense in $\cM$, see \cite{Par}.  Thus, by compactness of $\cR$
it suffices to show that $\{\rv(\mu_x): x\in \Per(f)\}\subset \cR$. Let $x\in \Per_n(f)$ for some $n\in \bN$. Recall that $\tau_x=x_1\cdots x_n$ denotes the generating segment of $x$. If $x=\Or(1)$ then $\rv(\mu_x)=w_\infty$.  Assume now that $x\not=\Or(1)$. Thus, at least one of the $x_i$'s in $\tau_x$ is not equal to $1$. Taking a different point in the (finite) orbit of $x$ if necessary, we may assume $x_n\not=1$. It follows from \eqref{defphipot} that the $y$-coordinate of $\Phi(\xi)$ is positive if and only if 
\begin{equation}
\xi\in \bigcup_{k\in \bN} \cC_{k+\lambda}(\xi^k).
\end{equation}
It follows that if $\tau_x$ does not contain a  subblock in $\{\tau_k: k\in \bN\}$ then $\rv(\mu_x)\in [0,a]\times \{0\}\subset \cR$. It remains to consider the case when $\tau_x$ contains at least one block in $\{\tau_k: k\in \bN\}$. By replacing $x$ with a point in the orbit of  $x$ if necessary, we can write $\tau_x$ as a finite concatenation of blocks of the form
\begin{equation}\label{eq55}
\tau_x=\eta_1\tau_{k_1}\, \cdots \,\eta_{l}\tau_{k_l},
\end{equation}
where the $\eta_i$'s are blocks that do not have a subblock contained in $\{\tau_k: k\in \bN\}$ and whose last symbol is either $0$ or $2$. The latter ensures that  the blocks $\tau_{k_i}$ are of maximal length.
We note that some of  the $\eta_i$'s in \eqref{eq55} may be the empty block. Let $n_i$ denote the length of $\eta_i$. For each $i$ there exists a $m_i$ such that $f^{m_i}(x)=\eta_i\tau_{k_i}\eta_{i+1}\tau_{k_{i+1}}\cdots$. We define $\upsilon_i=\frac{1}{n_i}\sum_{k=0}^{n_i-1} \Phi(f^{m_i+k}(x))$. It follows from the construction that $\upsilon_i\in [0,a]\times\{0\}$.
We conclude that
\begin{equation}\label{eq66}
\rv(\mu_x)=\frac{1}{n}\sum_{k=0}^{n-1}\Phi(f^k(x))=\frac{n_1}{n}\upsilon_1+ \frac{k_1+\lambda}{n}w_{k_1}+\cdots + \frac{n_l}{n}\upsilon_l+ \frac{k_l+\lambda}{n}w_{k_l}.
\end{equation} 
Notice that $n=l\lambda+\sum_{i=1}^l n_i +k_i$.
Therefore, \eqref{eq66} shows that $\rv(\mu_x)$ is a convex combination of points in $\cR$ which implies that $\rv(\mu_x)\in \cR$.
\end{proof}
\begin{corollary}
Let $k\in \bN$ and let $x\in \Per(f)$. Then $\rv(\mu_x)=w_k$ if and only if $\mu_x=\mu_{\xi^k}$.
\end{corollary}
\begin{proof}
The statement follows from \eqref{eq66}, $\upsilon_1,\cdots,\upsilon_l\in [0,a]\times\{0\}$ and the fact that $w_k$ is an extreme point of $\cR(\Phi)$. 
\end{proof}
We will make use of the following trivial facts that hold for all measure-preserving transformations.
\begin{lemma}\label{lem1}
Let $\mu\in \cM$, $A\subset X$ and $B\subset f^{-1}(A)$ then $\mu(B)\leq \mu(A)$, in particular $\mu(B)\leq \mu(f(B))$. Moreover, if $B\subset f^{-1}(A)$ then $\mu(B)=\mu(A)$ if and only if $\mu(f^{-1}(A)\setminus B)=0$.
\end{lemma}

We continue to use the notation from Proposition \ref{prop1} .  Recall that $\xi^k=\Or(1^{k+\lambda-1}2)$ and $\rv(\mu_{\xi^k})=w_k$.
\begin{proposition}\label{lem2}
Let $k\in \bN$ and $p=\frac{1}{k+\lambda}$. Let $\mu\in \cM$ with $\mu(\Phi^{-1}(w_0))=\lambda p$ and $\mu(\cC_{l+\lambda}(\xi^l))=p$ for $l=1,\cdots,k$. Then $\mu=\mu_{\xi^k}$.
\end{proposition}
\begin{proof}
We first notice that since the cylinders $\cC_{l+\lambda}(\xi^l), l=1,\cdots, k$ are pairwise disjoint, the assumptions of the proposition imply
\begin{equation}\label{eq111}
\mu\left(\Phi^{-1}\left(\{w_\infty\}\cup \bigcup_{l>k}\{v_l\}\cup \bigcup_{l\geq 1}\{u_l\}\right)\right)=0.
\end{equation}
Recall that $\tau_l=1^{l+\lambda-1}2$ denotes the generating segment of $\xi^l$.  We define cylinders $\cC^0=\cC_{1+k+\lambda}(2\tau_k)$ and $\cC^l=\cC_{l+1+k+\lambda}(1^l2\tau_k)$
for $1\leq l\leq \lambda-1$. It follows from the construction that the $\cC^i\cap \cC^j=\emptyset$ for all $0\leq i,j\leq  \lambda-1$ with $i\not=j$. Further, by Lemma \ref{lem1},
\begin{equation}\label{eq777}
p=\mu((\cC_{k+\lambda}(\tau_k))\geq \mu(\cC^0)\geq \cdots \geq \mu(\cC^{\lambda-1}).
\end{equation}
First, we prove the following.\\
Claim 1. $\mu(\cC^l)=p$ for all $0\leq l\leq \lambda-1$.\\
For the case $l=0$ we note that $\mu(\cC_{k+\lambda}(\xi^k))=\mu(\cC_{k+\lambda}(\tau_k))=p$. Moreover,
\begin{equation}\label{wert1}
f^{-1}(\cC_{k+\lambda}(\tau_k))=\cC^0\dot\cup \bigcup_{i=0}^1 \cC_{1+k+\lambda}(i\tau_k).
\end{equation}
Since $\Phi(\bigcup_{i=0}^1 \cC_{1+k+\lambda}(i\tau_k))=\{u_{k+1},v_{k+1}\}$  we may conclude from equation \eqref{eq111}  that $\mu(\bigcup_{i=0}^1 \cC_{1+k+\lambda}(i\tau_k))=0$. Therefore, 
the case $l=0$ follows from  \eqref{wert1} and Lemma \ref{lem1}.  
Clearly,
\begin{equation}
\cC_{k+\lambda+1}(2\tau_k),f^{-1}(\cC_{k+\lambda+1}(2\tau_k)),\cdots, f^{-(\lambda-1)}(\cC_{k+\lambda+1}(2\tau_k))
\end{equation}
 are pairwise disjoint sets with $\Phi \left(\bigcup_{r=1}^{\lambda-1}f^{-r}\left(\cC_{k+\lambda+1}(2\tau_k)\right)\right)=w_0$ satisfying $\mu\left(f^{-r}\left(\cC_{k+\lambda+1}(2\tau_k)\right)\right)=p$ for $r=0,\cdots,\lambda-1$. Hence
\begin{equation}\label{eq999}
\mu\left(\bigcup_{r=0}^{\lambda-1}f^{-r}\left(\cC_{k+\lambda+1}(2\tau_k)\right)\right)=\lambda p.
\end{equation}
Assume that the claim is false. Then, it follows from \eqref{eq777} that  $\mu(\cC^{\lambda-1})<p$. Since $\mu(f^{-(\lambda-1)}(\cC_{k+\lambda+1}(2\tau_k)))=p$, there exists $\eta=\eta_1\cdots \eta_{\lambda-1} 2$ such that $\cC_{k+2\lambda}(\eta\tau_k)\not=\cC^{\lambda-1}$ with $\mu(\cC_{k+2\lambda}(\eta\tau_k))>0$. Here $\cC_{k+2\lambda}(\eta\tau_k)\not=\cC^{\lambda-1}$ means that $\eta_i\not=1$ for some $i=1,\cdots,\lambda-1$. We conclude that there must exist a cylinder $\cC=\cC(\eta)$ of length $k+2\lambda+1$ contained in  $f^{-1}(\cC_{k+2\lambda}(\eta\tau_k))$ with $\mu(\cC)>0$.  Since $\eta_i\not=1$, $v_1\not\in\Phi(\cC)$. Moreover, since $\mu(\Phi^{-1}(u_1))=0$ we conclude that $\Phi(\cC)\not=\{u_1\}$. Hence  $\Phi(\cC)=\{w_0\}$. On the other hand, $\cC\cap \bigcup_{r=0}^{\lambda-1}f^{-r}\left(\cC_{k+\lambda+1}(2\tau_k)\right)=\emptyset$. Therefore \eqref{eq999} implies $\mu(\Phi^{-1}(w_0))>\lambda p$ with is a contraction. This proves Claim 1.\\
Next we define cylinder $\tC^0=\cC_1(2)$ and $\tC^l=\cC_{l+1}(1^l2)$ for  $1\leq l\leq \lambda-1$.
Claim 2. $\mu(\tC^l)=p$ for all $l=0,\cdots,\lambda-1.$\\
Obviously, $\C^l\subset \tC^l$ for all  $0\leq l\leq \lambda-1$. Thus, by Claim 1, $\mu(\tC^l)\geq p$. On the hand hand, we observe that $\tC^0, \cdots, \tC^{\lambda-1}$ are pairwise disjoint
cylinders with $\Phi(\tC^l)=w_0$ for all  $0\leq l\leq \lambda-1$.  Hence, $\sum_{l=0}^{\lambda-1}\mu(\tC^l)\leq \lambda p$. Putting these facts together proves the claim.\\
Claim 3. $\mu(\{\xi^k\})=p$.\\
The statement $\mu(\{\xi^k\})\leq p$ follows from  $\mu(\cC_{k+\lambda}(\xi^k))=p$ since  $\xi^k\in \cC_{k+\lambda}(\xi^k)$. Since$(\mu(\cC_{j+\lambda}(\xi^k)))_{j\geq 1}$ is a non-increasing sequence with limit $\mu(\{\xi^k\})$, it suffices to show that $\mu(\cC_{j(k+\lambda)}(\xi^k))\geq p$ for all $j\geq 1$. Note that the case $j=1$ is part of the assumption. Suppose on the contrary that there exists $j>1$ such that
 $\mu(\cC_{j(k+\lambda)}(\xi^k))< p$. Further, suppose $j$ is the smallest integer with this property. Recall that $\tau_k^{j-1}$ denotes the $(j-1)$-times concatenation of the block $\tau_k$. It follows that there exists a block $\eta=\eta_1\cdots \eta_{k+\lambda}$ with $\eta\not=\tau_k$ such at $\mu(\cC_{j(k+\lambda)}(\tau_k^{j-1}\eta))>0$. We conclude from Lemma \ref{lem1} that
 \begin{equation}\label{eqwich}
 \begin{split}
 0<\mu(\cC_{j(k+\lambda)}(\tau_k^{j-1}\eta))&\leq \mu\left(f^{(j-1)(k+\lambda)-1}(\cC_{j(k+\lambda)}(\tau_k^{j-1}\eta))\right)\\
 &=\mu(\cC_{1+k+\lambda}(2\eta)).
 \end{split}
 \end{equation}
Note that $\cC_{1+k+\lambda}(2\eta)$ and $\cC_{1+k+\lambda}(2\tau_k)$ are disjoint cylinders contained in $\cC_1(2)=\tC^0$. Since $\mu(\cC_{1+k+\lambda}(2\tau_k))=p$ (see Claim 1) we are able to deduce from $\eqref{eqwich}$ that $\mu(\cC_1(2))>p$ which is a contradiction to Claim 2 with $l=0$. This completes the proof of the Claim 3.
\\
To complete the proof of the proposition it remains to show that Claim 3 holds for all the points in the orbit of $\xi^k$.  Let $\zeta^k=f^l(\xi^k)$ for some $l=1,\cdots,k+\lambda-1$.  Since $f^{k+\lambda-l}(\zeta^k)=\xi^k$ we conclude from Lemma \ref{lem1} that $\mu(\{\zeta^k\})\leq \mu(\{\xi^k\})$. A similar argument shows $\mu(\{\xi^k\})\leq \mu(\{\zeta^k\})$. Hence $\mu(\{\zeta^k\})= \mu(\{\xi^k\})=p$, and the proof of the proposition is complete.
\end{proof}

\begin{theorem}\label{thmfin}
Let $\Phi$ be the potential defined in \eqref{defphipot}. Then $\cH(w_k)=0$ for all $k\in \bN$ and $\cH(w_\infty)=\log 2$. In particular, $w\mapsto \cH(w)$ is discontinuous at $w_\infty$. 

\end{theorem}
\begin{proof}
Since $w_\infty$ is an extreme point of $\cR(\Phi)$ as well as an extreme point of $\Phi(X)$, we obtain that $\rv(\mu)=w_\infty$ if and only $\mu(\Phi^{-1}(w_\infty))=1$. Note that $f\vert_{\Phi^{-1}(w_\infty)}=f\vert_{\{0,1\}^\bN}$ has topological entropy equal to $\log 2$. Therefore, $\cH(w_\infty)=\log 2$ is consequence of the variational principle for the entropy, see e.g. \cite{Wal:81}.\\
Fix $k\in \bN$ and let $\mu\in \cM$ with $\rv(\mu)=w_k$. Our goal is to show that $\mu=\mu_{\xi^k}$ which obviously suffices to prove the theorem.\\
Recall from \eqref{defphipot} that $\Phi^{-1}(v_l)=\cC_{l+\lambda}(1^{l+\lambda-1}2)$ which implies
\begin{equation}
f\left(\Phi^{-1}(v_{l+1})\right)=\Phi^{-1}(v_{l})
\end{equation}
for all $l\in \bN$. Therefore,  by Lemma \ref{lem1},
\begin{equation}\label{eqtt}
\cdots \leq \mu(\Phi^{-1}(v_{l+1}))\leq \mu(\Phi^{-1}(v_{l}))\leq \cdots\leq \mu(\Phi^{-1}(v_{2}))\leq  \mu(\Phi^{-1}(v_{1})).
\end{equation}
Next, we define several sets. We define $V_1(1)=\Phi^{-1}(v_{1})=\cC_{1+\lambda}(1^{\lambda}2)$ and $p=p_1(1)\eqdef\mu(V_1(1))$. Since $\rv(\mu)=w_k$, \eqref{defphipot} and \eqref{eqtt} imply $p_1(1)>0$. For $l\geq 2$ and $i=0,1,2$ we define $V_{l}(i)=\cC_{l+\lambda}(i1^{l+\lambda-2}2)$ and $p_{l}(i)=\mu(V_{l}(i))$. Since 
\begin{equation}
f^{-1}(V_l(1))=V_{l+1}(0)\dot\cup   V_{l+1}(1)  \dot\cup V_{l+1}(2)
\end{equation}
 we have $p_l(1)=p_{l+1}(0)+p_{l+1}(1)+p_{l+1}(2)$ for all $l\geq 1$. Moreover, $\Phi\vert_{V_l(0))}=u_{l}$, $\Phi\vert_{V_l(1))}=v_{l}$ and $\Phi\vert_{V_l(2))}=w_0$. Next we consider the  pre-images of $V_l(0)$. Fix $l\geq 2$ and let $j\in \bN$. We define 
 \begin{equation} 
 U_l^j=\{\xi\in f^{-j}(V_l(0)):\xi_1=2\, \, \mbox{ and } \, \, \xi_2,\cdots, \xi_j\in \{0,1\} \},
 \end{equation}
 that is
 \begin{equation}\label{equn1}
 U_l^j=\bigcup_{i_2,\cdots, i_j\in \{0,1\}} \cC_{j+l+\lambda} (2i_2\cdots i_j 01^{l+\lambda-2}2)\subset f^{-j}(V_l(0)).
 \end{equation}
We note that \eqref{equn1} represents $U_l^j$ as a pairwise disjoint union of sets. Moreover, 
$U_l^j\cap U_{l'}^{j'}=\emptyset$ for all $l,l'\in\bN$   and all $j,j'\geq 1$  whenever $l\not=l'$ or $j\not=j'$ or both.
We define $U_l=\bigcup_{j\geq 1} U_l^j$ and claim the following.\\
Claim 1. $\mu\left(U_l\right) = \mu(V_l(0))=p_l(0)$.\\
To prove the claim we consider sets 
\begin{equation}
\tU_l(j)=\bigcup_{i_1,\cdots, i_j\in \{0,1\}} \cC_{j+l+\lambda} (i_1\cdots i_j 01^{l+\lambda-2}2)\subset f^{-j}(V_l(0)).
\end{equation}
It follows that $(\tU_l(j))_{j\geq 1}$ is a  sequence of pairwise disjoint sets. Hence, $\lim_{j\to\infty}\mu(\tU_l(j))=0$. Further, by construction,
\begin{equation}
\mu(V_l(0))-\mu(\tU_l(j))= \mu(f^{-j}(V_l(0)))-\mu(\tU_l(j))=\sum_{i=1}^{j} \mu(U_l^i).
\end{equation}
Therefore, Claim 1 follows from $\mu(U_l)=\sum_{i=1}^\infty \mu(U_l^i)$.\\
Claim 2. $\mu(\Phi^{-1}(w_0))\geq \lambda p$.\\
To prove the claim we   first construct $\tY_1=\tY_1(\mu)\subset X$ with $\mu(\tY_1)=p$ such that for all $\xi\in \tY_1$ we have $\xi_1=2$ and $\xi_2,\cdots,\xi_\lambda\not=2$. 
By construction, $(V_l(1))_{l\geq 1}$ is  a sequence of pairwise disjoint sets.
Thus, $\lim_{l\to\infty}\mu(V_l(1))=0$. By applying that the sets
\begin{equation}
V_2(0),V_2(2),V_3(0),V_3(2),\cdots,V_l(0),V_l(2),\cdots 
\end{equation}
 are pairwise disjoint, 
 a similar argument
as in the proof of Claim 1 shows
\begin{equation}\label{eq38}
\mu\left(\bigcup_{l\geq 2} V_l(0) \cup V_l(2)\right)=\sum_{l=2}^{\infty} \left[\mu(V_l(0)) + \mu(V_l(2))\right]=V_1(1)=p.
\end{equation}
Evidently, $U_l\cap V_{l'}(2)=\emptyset$ for all $l,l'\geq 2$. Therefore, we may conclude from \eqref{eq38}  and Claim 1 that
\begin{equation}
\mu\left(\bigcup_{l\geq 2} U_l \cup V_l(2)\right)=p.
\end{equation}
We define $\tY_1=\bigcup_{l\geq 2} U_l \cup V_l(2)$. By construction, if $\xi\in \tY_1$ then $\xi_1=2$ and $\xi_2,\cdots,\xi_\lambda\not=2$. Define $\tY_2= f^{-1}(\tY_1),\cdots,\tY_\lambda= f^{-\lambda-1}(\tY_1)$. It follows that
$\tY_1,\cdots,\tY_\lambda$ are pairwise disjoint sets with $\mu(\tY_1)=\cdots =\mu(\tY_\lambda)=p$. 
We define
\begin{equation}
Y_0(\lambda)=\tY_1\cup\cdots\cup\tY_\lambda.
\end{equation}
Hence, $\mu(Y_0(\lambda))=\lambda p$. We note that $\Phi(\xi)=w_0$ for all $\xi\in Y_0(\lambda)$ which completes the proof of Claim 2.\\
Next we  compute  $\rv(\mu)$ by integrating $\Phi$ over various subsets of $X$. Define
\begin{equation}
Y_{k}=\bigcup_{l=1}^k V_l(1) = \Phi^{-1}\left(\{v_1,\cdots,v_k\}\right)\, \, \mbox{ and }\, \, Y_{gk}=  \Phi^{-1}\left(\bigcup_{l>k}\{v_l\}\right)
\end{equation}
and
\begin{equation}
Y_0=\Phi^{-1}([0,a]\times\{0\})\setminus Y_0(\lambda).
\end{equation}
 Moreover, define $p_{gk}=\mu(Y_{gk})$ and $p_0=\mu(Y_0)$.
Thus $X=Y_0(\lambda)\cup Y_k\cup Y_{gk}\cup Y_{0}$ is a union of pairwise disjoint sets. To compute $\int_{Y_0(\lambda)\cup Y_k} \Phi\, d\mu$ we define $p_k=p_k(1)$ and $p_l=p_{l}(1)-p_{l+1}(1)$ for $l=k-1,\cdots, 1$. Hence $p=\sum_{l=1}^kp_l$. By making a telescope sum argument  and applying \eqref{eq2} we obtain
\begin{equation}\label{r1}
\begin{split}
\int_{Y_0(\lambda)\cup Y_k} \Phi\, d\mu & = \lambda p w_0+\sum_{l=1}^k p_l(1) v_l\\ 
&=  \sum_{l=1}^k p_l\left(\lambda w_0+\sum_{i=1}^l v_i\right)  \\
&= \sum_{l=1}^k  p_l(l+\lambda) w_l.
\end{split}
\end{equation}
Evidently we have $p_l\leq 1/(l+\lambda)$ for $l=1,\cdots,k$. If  $\mu(Y_{gk})\not=0$  we define
\begin{equation}\label{r2222}
\int_{Y_{gk}} \Phi \, d\mu =\sum_{l=k+1}^\infty p_l(1) v_l =\mu(Y_{gk})\sum_{l=k+1}^\infty \frac{p_l(1)}{p_{gk}}  v_l\eqdef \mu(Y_{gk}) v_{\infty},
\end{equation}
otherwise we set $v_\infty=0$.
Similarly, if $\mu(Y_0)\not=0$, we define 
\begin{equation}\label{r3}
\int_{Y_0} \Phi\, d\mu = \mu(Y_0)\left(\frac{1}{\mu(Y_0)} \int_{Y_0} \Phi\, d\mu \right)\eqdef \mu(Y_0) \omega_0,
\end{equation}
otherwise we set $\omega_0=0$.
It follows from the definition of $Y_0$ that $\omega_0\in [0,a]\times \{0\}$. 
Let $\ell_1$ denote the line through $w_{k}$ and $w_{k-1}$, and let $\ell_2$ denote the line through $v_{k+1}$ and $w_{k}$. 
Let $\cG$ denote the intersection of the two closed half-spaces  of points on and below the lines $\ell_1$ and $\ell_2$ respectively. 
Clearly $\cG$ is convex.
By definition, the closed line segment $[w_{k},w_{k-1}]$ is contained in $\ell_1$.  Moreover,
by \eqref{eq2} the closed line segment $[w_{k+1},w_k]$ is contained in $\ell_2$. 
We conclude that $\ell_i$ is the supporting hyperplane of the face $[w_{k+i-1},w_{k+i-2}]$ of $\cR(\Phi)$. This shows  that $\cR(\Phi)\subset \cG$.
It follows from Proposition \ref{prop000} (ii) that  
 $w_{k}$ is an extreme point of $\cG$. Moreover, $\{v_l: l\geq k+1\}\subset \cG$ which together with \eqref{r2222} implies that $v_\infty\in \cG$.
It follows from
\begin{equation}
\rv(\mu)=\int_{Y_0(\lambda)\cup Y_k} \Phi\, d\mu+\int_{Y_{gk}} \Phi\, d\mu=\int_{Y_0} \Phi\, d\mu
\end{equation}
and equations \eqref{r1},\eqref{r2222} and \eqref{r3} that $\rv(\mu)$ is a convex combination of  the points $w_1,\cdots,w_k,v_\infty, \omega_0$ all of which belong to $\cG$.
By using that $w_k=\rv(\mu)$ is an extreme point of $\cG$  we may conclude that this convex combination must coincide with $w_k$ itself. Hence $p_1=\cdots=p_{k-1}=0$, $\mu(Y_0)=  \mu(Y_{gk})=0$ and $p=p_1(1)=p_2(1)=\cdots=p_k(1)=p_k=\frac{1}{k+\lambda}$. 
This shows that  $\mu$ satisfies the assumptions of Proposition \ref{lem2}. Thus, by Proposition   \ref{lem2} we have $\mu=\mu_{\xi^k}$ which completes the
proof of the theorem.
\end{proof}

\end{document}